\documentclass[a4paper]{amsart}
\usepackage{ogonek}
\usepackage{amsmath,amsfonts,amssymb,amsthm,amscd,color}
\usepackage{latexsym,graphicx}
\usepackage[matrix,arrow,ps,color,line,curve,frame]{xy}
\usepackage[T1]{fontenc}
\usepackage{stmaryrd}

\usepackage{stmaryrd}

\newcommand{\id}{\mathrm{id}}
\newcommand{\A}{\mathcal{A}}
\newcommand{\cH}{\mathcal{H}}
\newcommand{\C}{\mathcal{C}}
\newcommand{\B}{\mathcal{B}}
\newcommand{\cO}{\mathcal{O}}
\newcommand{\cP}{\mathcal{P}}

\renewcommand{\[}{\begin{equation}}
\renewcommand{\]}{\end{equation}}

\newtheorem{theorem}{Theorem}[section]

\newtheorem{lemma}[theorem]{Lemma}
\newtheorem{corollary}[theorem]{Corollary}
\theoremstyle{definition}
\newtheorem{definition}[theorem]{Definition}
\theoremstyle{remark}

\numberwithin{equation}{section}

\begin{document}
\parskip=0.75\baselineskip
\parindent=4mm

\author{Piotr~M.~Hajac}
\address{Instytut Matematyczny, Polska Akademia Nauk, ul.~\'Sniadeckich 8, 00--656 Warszawa, Poland}
\email{pmh@impan.pl}
\author{Tomasz Maszczyk}
\address{Instytut Matematyki, 
     Uniwersytet Warszawski,
ul.\ Banacha 2,
02--097 Warszawa, Poland}
\email{t.maszczyk@uw.edu.pl}
\title[Pulling back noncommutative bundles]{\vspace*{-25mm}\small
Pullbacks and nontriviality of\\
  associated noncommutative vector bundles}
\maketitle
\vspace*{-10mm}

\begin{abstract}\vspace*{-2.5mm}
Our main theorem is that the pullback of an associated noncommutative vector bundle induced by 
an equivariant map of quantum principal bundles is a noncommutative vector bundle associated 
via the same finite-dimensional representation of the structural quantum group. On the level of $K_{0}$-groups, 
we realize the induced map by the pullback of explicit matrix idempotents.  
We also show how to extend our result to the case when the quantum-group representation  is infinite dimensional, and then
 apply it to the Ehresmann-Schauenburg
quantum groupoid.
Finally, using noncommutative Milnor's join construction, we define quantum quaternionic projective spaces together with noncommutative
tautological quaternionic line bundles and their duals.
As a key application of the main theorem, we show that these bundles are stably non-trivial as noncommutative complex vector
bundles. 
\end{abstract}
\tableofcontents
\vspace*{-5mm}\noindent
Our result is motivated by the search of $K_0$-invariants. The main idea is to use equivariant homomorphisms to facilitate 
computations of such invariants by moving them from more complicated to simpler algebras. This strategy was recently successfully
applied in \cite{hpsz} to distinguish the $K_0$-classes of noncommutative
 line bundles over two different types of quantum complex projective spaces.
Herein we\footnote{
\hspace*{6mm}\emph{AMS subject classification codes:} 46L80, 46L85, 58B32.} 
generalize from associated noncommutative line bundles to associated noncommutative 
\emph{vector} bundles. Then we apply our general theorem to 
noncommutative
 vector bundles associated with a finitely iterated 
equivariant noncommutative join \cite{dhh15} of 
$SU_q(2)$ with itself.

Recall that  the $n$-times iterated join $SU(2)*\dots*SU(2)$ yields the odd sphere $S^{4n+3}$, and the diagonal
$SU(2)$-action on the join  defines  a fibration producing the $n$-th quaternionic projective space: $S^{4n+3}/SU(2)=\mathbb{HP}^n$.
Therefore, we denote the $n$-iterated equivariant noncommutative join of $C(SU_q(2))$ with itself by $C(S^{4n+3}_q)$, and
consider the fixed-point subalgebra $C(S^{4n+3}_q)^{SU_q(2)}$ as the defining C*-algebra $C(\mathbb{HP}^n_q)$ of a
\emph{quantum quaternionic projective space}. Then we define the noncommutative tautological quaternionic line bundle and its dual
as noncommutative complex vector bundles associated through the contragredient representation of the fundamental represention of~$SU_q(2)$
and  the fundamental represention itself, respectively. We observe that, as in the classical case, they are isomorphic as noncommutative (complex) 
vector bundles. 
 
A classical argument, proving the non-triviality of a vector bundle associated with a principal bundle by restricting 
the vector bundle
 to an appropriate subspace, uses 
the fact that  the thus restricted vector bundle is associated with the  restricted principal bundle. 
The latter restriction is encoded by an equivariant map of total spaces of principal bundles, which induces a 
natural transformation of Chern characters 
of the vector bundles in question. 

More precisely,  let $X$ and $X'$ be compact Hausdorff right $G$-spaces. Assume that the \mbox{$G$-action} on $X$ is free. 
If $f:X'\to X$ is a continuous $G$-equivariant map,  the \mbox{$G$-action} on $X'$ is automatically free as well, and therefore  
we have then a $G$-equivariant homeomorphism of compact principal bundles over $X'/G$:
\[\label{equiv}
X'\ni x'\longmapsto \big(x'G,f(x')\big)\in X'/G\underset{X/G}{\times}X.
\]
Its inverse, given by means of the \emph{translation map} $\tau:  X\underset{X/G}{\times}X\to G$, $\tau(x, xg)=g$, is as follows:
\[
 X'/G\underset{X/G}{\times}X\ni \big(x'G, x\big)\longmapsto x'\tau\big(f(x'), x\big)\in X'.
\]
Here the fiber product over $X/G$  is given by the canonical quotient map $X\to X/G$, $x\mapsto xG$, and
by the $f$-induced map $f/G: X'/G\to X/G$, $xG\mapsto f(x)G$. 

Therefore,  $G$-equivariant continuous maps between total spaces of compact principal $G$-bundles are  equivalent to continuous maps between 
their base spaces. In particular, the isomorphism class of a compact principal $G$-bundle is uniquely determined by the homotopy class of a map 
from its base space to a compact approximation of the classifying space $BG$. As shown by Milnor~\cite{m-j56}, 
the quotient map $G*\dots*G\to (G*\dots*G)/G$ 
is a principal $G$-bundle which is a compact approximation of a universal principal $G$-bundle $EG\to BG$.    

To decide whether a given compact principal $G$-bundle is nontrivial, it is sufficient to prove that at least one of its associated vector bundles is 
nontrivial. 
Furthermore, every associated vector bundle is a pullback of a universal vector bundle, both corresponding to the same 
representation  $G\to GL(V)$. This is a 
consequence of the  compatibility  of associating and  pulling back
\[
X'\overset{G}{\times} V=\Big(X'/G\underset{X/G}{\times}X\Big)\overset{G}{\times} V
=X'/G\underset{X/G}{\times}\Big(X\overset{G}{\times} V\Big)
=(f/G)^*\Big(X\overset{G}{\times} V\Big),
\]
which is afforded by the $G$-equivariant homeomorphism (\ref{equiv}).

In particular, for $G=SU(2)$ (when $BG=\mathbb{HP}^\infty$), the restriction  of the 
tautological quaternionic line bundle $\tau_{{}_{\mathbb{HP}^n}}$ from $\mathbb{HP}^n$ to $\mathbb{HP}^1$ is the 
tautological quaternionic line bundle $\tau_{{}_{\mathbb{HP}^1}}$ over $\mathbb{HP}^1$, so that
the Chern character computation proving the nontriviality of  $\tau_{{}_{\mathbb{HP}^1}}$ proves also the nontriviality of
$\tau_{{}_{\mathbb{HP}^n}}$.

The present paper generalizes this reasoning to the noncommutative setting as follows. First, we use the Gelfand-Naimark theorem
to encode compact Hausdorff spaces as commutative unital C*-algebras. Next, we employ the Peter-Weyl theory to describe a compact group
as a Hopf algebra of representative functions. Then, we take advantage of the Peter-Weyl theory extended from compact groups to 
compact principal bundles \cite{bh14} to express a compact principal bundle as a comodule algebra over the Hopf algebra  
of representative functions. Finally, we use the Serre-Swan theorem to encode a vector bundle as a finitely generated projective module.
In particular, we describe an associated vector bundle as an associated finitely generated projective module \cite{bhms}
using the Milnor-Moore cotensor product~\cite{mm65}.
Having all these basic structures given in terms of commutative algebras, we generalize by dropping the assumption of commutativity.

In precise technical terms, we proceed as follows.
For any finite-dimensional corepresentation $V$ of a coalgebra $\mathcal{C}$
coacting principally on an algebra $\mathcal{A}$, we can use the cotensor product $\Box^\C$ to form an associated 
finitely generated projective module $\mathcal{A}\Box^\C V$ over the coaction-invariant subalgebra~$\mathcal{B}$. 
The module $\mathcal{A}\Box^\C V$ is the section module of the associated 
noncommutative vector bundle. 
If $\mathcal{A}'$ is an algebra with a principal coaction of~$\mathcal{C}$, and $\mathcal{B'}$ is its coaction-invariant subalgebra, then
any equivariant (colinear) algebra homomorphism $\mathcal{A}\to\mathcal{A}'$
restricts and corestricts to an algebra homomorphism $\mathcal{B}\to\mathcal{B}'$ making $\mathcal{B}'$ a 
$(\mathcal{B}'-\mathcal{B})$-bimodule. The theorem of the paper is:

\noindent{\bf Theorem 0.1.} \emph{The finitely generated
left $\mathcal{B}'$-modules $\mathcal{B}'\otimes_\mathcal{B}(\mathcal{A}\Box^\C  V)$ and $\mathcal{A}'\Box^\C  V$
are \emph{isomorphic}. 
In particular, for any equivariant *-homomorphism $f:A\to A'$ between unital C*-algebras
equipped with a free action of a compact quantum group, 
the induced K-theory map $f_*\colon K_0(B)\to K_0(B')$, where $B$ and $B'$ are the respective fixed-point subalgebras,
satisfies $f_*([A\Box^\C  V])=[A'\Box^\C  V]$. }

We begin by stating our main result in the standard and easily accessible Hopf-algebraic setting (Theorem~\ref{hopf}). Then we state 
and prove two slightly different 
 coalgebraic generalizations of the result: Theorem~\ref{faith} based on faithful flatness and coflatness (cosemisimple coalgebras but
with  corepresentations of any dimension), and Theorem~\ref{main} based on Chern-Galois 
theory~\cite{chg} (arbitrary coalgebras 
but with corepresentations of finite dimensions). 

An advantage of the faithful-flatness-and-coflatness approach is a possibility to apply it in the case of infinite-dimensional 
vector fibers. In particular, we prove (Theorem~\ref{twofaith}) 
that the pullback of the Ehresmann-Schauenburg quantum groupoid of a quantum principal bundle 
is  the  Ehresmann-Schauenburg quantum groupoid of the pullback of the quantum principal bundle.  

An advantage of the Chern-Galois approach is a possibility to compute explicitly a $K$-theoretic invariant whose non-vanishing 
for a given principal extension
proves
that  the extension is not cleft~\cite[\S 8.2]{cleft}. 
The aforementioned  coalgebraic generality is necessary  
already in the case of celebrated 
non-standard Podle\'s spheres, 
which are fundamental examples of noncommutative geometry going beyond the reach of Hopf-Galois 
theory~\cite{p-p87,brz96,brz97,bh99,sch-sch}. 

Finally, we use the Peter-Weyl functor  to make the result applicable to free actions of compact quantum groups 
on unital C*-algebras~\cite{free}.
In the C*-algebraic setting we consider our example and main application:
 the stable non-triviality of the noncommutative tautological quaternionic line bundles and their duals.

\section{Pushing forward modules associated with Galois-type coactions}

Let $\C$ be a coalgebra, $\delta_M\colon M\to M\otimes\C$ a right coaction, 
and ${}_N\delta\colon N\to \C\otimes N$ a left coaction. 
The \emph{cotensor} product of $M$ with $N$ is
$
M\Box^\C N:=\ker(\delta_M\otimes\id-\id\otimes {}_N\delta)
$. In what follows, we will also use the Heyneman-Sweedler notation (with the summation sign suppressed) for  
 comultiplications and right coactions:
\begin{equation}
\Delta(c)=:c_{(1)}\otimes c_{(2)},\quad\delta_M(m)=:m_{(0)}\otimes m_{(1)}\,.
\end{equation}

Next,
let $\cH$ be a Hopf algebra with bijective antipode $S$, comultiplication~$\Delta$, and counit~$\varepsilon$.
Also, let  $\delta_{\A}\colon\A\to\A\otimes\cH$ be a coaction rendering $\A$ a right $\cH$-comodule
algebra. The subalgebra of coaction invariants $\{b\in\A\;|\;\delta_{\A}(b)=b\otimes 1\}$ is called the coaction-invariant (or 
fixed-point) subalgebra.
We say that $\A$ is a \emph{principal} comodule algebra iff there exists a 
\emph{strong connection} \cite{h-pm96,dgh01,chg}, i.e., 
a unital linear map $\ell :\mathcal{H} \rightarrow \A \otimes\A$ satisfying:
\begin{enumerate}
\item[(1)] 
$(\mathrm{id}\otimes \delta_{\A}) \circ 
\ell = (\ell \otimes \mathrm{id}) \circ \Delta$,
$\big(((S^{-1}\otimes\id)\circ\mathrm{flip}\circ\delta_{\A}) \otimes \mathrm{id}\big) \circ 
\ell = (\mathrm{id} \otimes \ell) \circ
\Delta$;
\item[(2)] 
$m \circ \ell=\varepsilon$, where 
$m\colon \A\otimes\A\to \A$ is the multiplication map.
\end{enumerate}

Let $\cH$ be a Hopf algebra with bijective antipode. In \cite{piece}, the principality of an 
$\cH$-comodule algebra was defined by requiring the bijectivity of the canonical map
(see Definition~\ref{def.principal} (1)) and equivariant projectivity (see Definition~\ref{def.principal2} (2)).
One can  prove (see \cite{bh} and references therein) that
an $\cH$-comodule algebra is principal in this sense if and only if it admits a strong 
connection. Therefore, we will treat the existence of a strong connection 
as a condition defining the principality of a comodule algebra and avoid the original
definition of a principal comodule algebra. The latter is important when going beyond coactions that are
algebra homomorphisms --- then the existence of a strong connection is implied by
principality but we do not have the reverse implication~\cite{chg}.

\begin{theorem}\label{hopf}
Let $\A$ and $\A'$ be right $\cH$-comodule algebras for a Hopf algebra $\cH$ with bijective antipode, 
and $V$ be a finite-dimensional
left $\cH$-comodule. Denote by $\B$ and $\B'$
the respective coaction-invariant subalgebras. Assume that $\A$ is principal and that there
exists an $\cH$-equivariant algebra homomorphism $f\colon\A\to\A'$. The restriction-corestriction of $f$ to $\B\to\B'$ makes
$\B'$ a \mbox{$(\B'-\B)$-bi}module such that the associated left $\B'$-modules $\B'\otimes_B(\A\Box^\cH V)$ and $\A'\Box^\cH V$ are 
\emph{isomorphic}.
In particular, the induced map 
$$
f_*\colon K_0(\B)\longrightarrow K_0(\B')\quad
\text{satisfies}\quad f_*([\A\Box^\cH V])=[\A'\Box^\cH V].
$$
\end{theorem}

As will be explained later on, the above Theorem~\ref{hopf} specializes Theorem~\ref{main},
and Theorem~\ref{cor} is a common denominator of Theorem~\ref{faith}  and
Theorem~\ref{hopf}.

\subsection{Faithfully flat coalgebra-Galois extensions}

\vspace*{-1.5mm}
\begin{definition}
\label{def.principal}\cite{bh}
Let $\mathcal{C}$ be a coalgebra coaugmented by  a group-like element $e\in \mathcal{C}$, and 
$\mathcal{A}$ an algebra and a right 
$\mathcal{C}$-comodule via
$\delta_{\mathcal{A}}:\mathcal{A}\rightarrow \mathcal{A}\otimes \mathcal{C}$ such that 
$\delta_{\mathcal{A}}(1)=1\otimes e$. 
Put
$$
\B:=\{b\in \mathcal{A}~|~\forall\,a\in\mathcal{A}\colon\delta_{\mathcal{A}} (ba)=b\delta_{\mathcal{A}} (a)\}
$$
 (coaction-invariant subalgebra). 
We say that the inclusion $\B\subseteq \mathcal{A}$ is an
 {\em e-coaugmented $\mathcal{C}$-Galois extension}
 iff
\begin{enumerate}
\item[(1)] the \emph{canonical map}
$
can: \mathcal{A}\otimes_\B\mathcal{A}{\rightarrow} \mathcal{A}\otimes 
\mathcal{A},\;a\otimes a'\mapsto a\delta_{\mathcal{A}} (a')
$
is bijective,
\item[(2)]  $\delta_{\mathcal{A}}(1)=1\otimes e$.
\end{enumerate}
\end{definition}

With any $\mathcal{C}$-Galois extension one can associate
the \emph{canonical entwining}~\cite{bh99}:
\begin{equation}
\psi:\mathcal{C}\otimes \mathcal{A}{\longrightarrow} \mathcal{A}\otimes \mathcal{C}, \quad c\otimes a\longmapsto
can(can^{-1}(1\otimes c)a).
\end{equation}
When the $\mathcal{C}$-Galois extension is $e$-coaugmented, then combining \cite[Proposition~2.2]{bm98} with 
\cite[Proposition~4.2]{bh99} yields the following presentations of the coaction-invariant subalgebra:
\begin{equation}\label{super}
\B=\{b\in \mathcal{A}~|~\delta_{\mathcal{A}} (b)=b\otimes e\}=\{b\in \mathcal{A}~|~\psi(e\otimes b)=b\otimes e\}.
\end{equation} 

\begin{theorem}\label{faith}
Let $\B\subseteq \mathcal{A}$ and $\B'\subseteq \mathcal{A}'$ be  
$e$-coaugmented $\mathcal{C}$-Galois extensions, let $V$ be a left 
$\mathcal{C}$-comodule. 
Assume that $\mathcal{A}'$ is faithfully flat as a right $\B'$-module and that coalgebra $\mathcal{C}$ is  cosemisimple. 
Then  every $\mathcal{C}$-equivariant algebra map \mbox{$f:\mathcal{A}\rightarrow \mathcal{A}'$} restricts and corestricts to an 
algebra homomorphism 
$\B\rightarrow\B'$, and induces an isomorphism 
\begin{equation*}
\B'{\otimes}_{\B}(\mathcal{A}\Box^{\mathcal{C}} V)\cong
 \mathcal{A'}\Box^{\mathcal{C}} V
\end{equation*}
of left $\B'$-modules that is natural in~$V$.  
\end{theorem}
\begin{proof}  
Since $\mathcal{A'}$ is a faithfully flat right $\B'$-module, the map of left $\B'$-modules right $\mathcal{C}$-comodules
\begin{equation}
\widetilde{f}:=m_{\mathcal{A'}}\circ(\id_{\mathcal{B'}}\otimes_{\B} f)\colon
\B'\otimes_{\B}\mathcal{A}\longrightarrow \mathcal{A'},\quad m_{\mathcal{A'}}(b'\otimes a')=b'a',
\end{equation}
is an isomorphism if and only if the map of left $\mathcal{A'}$-modules right $\mathcal{C}$-comodules
\begin{equation}
\id_{\mathcal{A'}}\otimes_{\B'}\widetilde{f}\colon
\mathcal{A'}\otimes_{\B'}\B'\otimes_{\B}\mathcal{A}\longrightarrow
\mathcal{A'}\otimes_{\B'}\mathcal{A'}
\end{equation}
is an isomorphism. Replacing the left-hand-side $\B'$ by  $\mathcal{A}$, the latter is an isomorphism if and only if
\begin{equation}
\id_{\mathcal{A'}}\otimes_{A}( f\otimes_{\B'} f)\colon
\mathcal{A'}\otimes_{\mathcal{A}}\mathcal{A}\otimes_{\B}\mathcal{A}\longrightarrow
\mathcal{A'}\otimes_{\B'}\mathcal{A'}
\end{equation}
is an isomorphism.  Thus, from the commutativity of the diagram
\begin{equation}
\xymatrixcolsep{5pc}\xymatrix{
\mathcal{A'}\otimes_{\mathcal{A}}\mathcal{A}\otimes_{\B}\mathcal{A}\quad 
\ar[d]_-{\id_{\mathcal{A'}}\otimes_{\mathcal{A}}\hspace{0.2em} can}
\ar[r]^-{\id_{\mathcal{A'}}\otimes_{\mathcal{A}}( f\otimes_{\B'} f)}&
\quad\mathcal{A'}\otimes_{\B'}\mathcal{A'}\ar[d]^{can'}\\
\mathcal{A'}\otimes_{\mathcal{A}}\mathcal{A}\otimes\mathcal{C} \ar[r]^-\cong &
\mathcal{A'}\otimes\mathcal{C}}
\end{equation}
and the bijectivity of the canonical maps, we infer that $\widetilde{f}$ is an isomorphism.

Furthermore, as $\widetilde{f}$ is a homomorphism of left $\B'$-modules 
right $\mathcal{C}$-comodules, we conclude that
\begin{equation}
\widetilde{f}\hspace{0.2em}\Box^{\mathcal{C}}\id_{V}: 
(\B'\otimes_{\B}\mathcal{A})\Box^{\mathcal{C}} V\longrightarrow
 \mathcal{A'}\Box^{\mathcal{C}} V
\end{equation}
is an isomorphism of left $\B'$-modules. Finally, since $\mathcal{C}$ is cosemisimple, and any comodule
over a cosemisimple coalgebra is injective \cite[Theorem 3.1.5 (iii)]{hopf}, whence coflat \cite[Theorem 2.4.17 
(i)-(iii)]{hopf}, the balanced tensor product $\B'\otimes_{\B}(-)$ and the cotensor product 
$(-)\Box^{\mathcal{C}}V$ commute. Therefore, there is a natural in $V$ isomorphism of left $\B'$-modules
\begin{equation}
\B'{\otimes}_{\B}(\mathcal{A}\Box^{\mathcal{C}} V)\cong 
(\B'{\otimes}_{\B}\mathcal{A})\Box^{\mathcal{C}} V\cong
 \mathcal{A'}\Box^{\mathcal{C}} V,
\end{equation}
as claimed.
\end{proof}

Note that the algebra $\mathcal{A}$ in the above theorem is given as a right $\mathcal{C}$-comodule. However, it
also enjoys a natural  left $\mathcal{C}$-comodule structure provided that
 the canonical entwining
is bijective. Indeed, one can then define a left coaction as follows:
\begin{equation}\label{leftco}
_\mathcal{A}\delta:\mathcal{A}\longrightarrow \mathcal{C}\otimes \mathcal{A},\quad
_\mathcal{A}\delta(a)=\psi^{-1}(a\otimes e).
\end{equation}
Another consequence of invertibility of the canonical entwining $\psi$ and \eqref{super} is the equality
  $\B=\{b\in\A\;|\;  _{\A}\delta(b)=e\otimes b\}$ and the fact that the left comultiplication $_{\A}\delta$ is  right $\B$-linear.

The left and right $\mathcal{C}$-comodule structures  on  $\mathcal{A}$, together with the left  $\B$-linearity of 
$\delta_\mathcal{A}$ and the right $\B$-linearity of $_\mathcal{A}\delta$, allow us to construct a  
$\mathcal{B}$-bimodule  $\mathcal{A}\Box^{ \mathcal{C}} \mathcal{A}$. 
 In the Hopf-Galois case, it is the Ehresmann-Schauenburg quantum groupoid~\cite{s-p96}, 
which is a noncommutative generalization of the  Ehresmann  groupoid of a principal bundle~\cite{p-j84}. 

Now we want to apply Theorem~\ref{faith} to $\mathcal{A}\Box^{ \mathcal{C}} \mathcal{A}$ with the right 
$\mathcal{A}$ viewed as a left \mbox{$\mathcal{C}$-comodule}. To this end, we need the following:
\begin{lemma}\label{cool}
Let $\A$ and $\A'$ be  $e$-coaugmented $\C$-Galois extensions with invertible canonical entwinings. Then, if  
an algebra map $f\colon\A\to\A'$ is  right $\C$-colinear, it is also left   $\C$-colinear.
\end{lemma}
\begin{proof}
If  $f\colon\A\to\A'$ is a $\C$-colinear algebra homomorphism, 
then it intertwines the canonical maps $can$ and $can'$ of $\A$ and $\A'$
 respectively  in the following way:
\begin{equation}
(f\otimes \id_{\C})\circ can=can'\circ (f\otimes_{\B} f).
\end{equation}
Since both $can$ and $can'$ are invertible, this implies that
\begin{equation}
(can')^{-1}\circ (f\otimes \id_{\C})=(f\otimes_{\B} f)\circ can^{-1}.
\end{equation}
Therefore, canonical entwinings $\psi$ and $\psi'$ are related as follows:
\begin{align}
\big((f\otimes \id_{\C})\circ \psi\big)(c\otimes a) 
            &= (f\otimes \id_{\C})\Big( can\big( (can^{-1}(1\otimes c)a)\big)\Big) \nonumber\\
            &= \big(can'\circ (f\otimes_{\B}f)\big)\big(can^{-1}(1\otimes c)a\big)\nonumber\\
            &= can'\big( (f\otimes_{\B}f)\big(can^{-1}(1\otimes c)a\big)\big)\nonumber\\
            &= can'\big( (f\otimes_{\B}f)\big(can^{-1}(1\otimes c)\big)f(a)\big)\nonumber\\
            &= can'\Big( \big((can')^{-1}(f(1)\otimes c)\big)f(a)\Big)\nonumber\\
            &= can'\Big( \big((can')^{-1}(1\otimes c)\big)f(a)\Big)\nonumber\\
            &= \psi'(c\otimes f(a))\nonumber\\
            &= \big(\psi'\circ (\id_{\C}\otimes f)\big)(c\otimes a).
\end{align}
As $c$ and $a$ are arbitrary, we conclude that
\begin{equation}
(f\otimes \id_{\C})\circ \psi=\psi'\circ (\id_{\C}\otimes f).
\end{equation}
Now it follows from the invertibility of the entwinings that
\begin{equation}
(\psi')^{-1}\circ (f\otimes \id_{\C})=(\id_{\C}\otimes f)\circ \psi^{-1}.
\end{equation}
Finally, evaluating the above equation on $a\otimes e$, we get
\begin{equation}
\big((\psi')^{-1}\circ (f\otimes \id_{\C})\big)(a\otimes e)=\big((\id_{\C}\otimes f)
\circ \psi^{-1}\big)(a\otimes e),
\end{equation}
which reads
\begin{equation}
\big(_{\mathcal{A}'}\delta\circ f\big)(a)=\big((\id_{\C}\otimes f)\circ _\A\!\delta\big)(a).
\end{equation}
Since $a$ is arbitrary, we infer the desired left $\C$-colinearity of~$f$.
\end{proof}

Note that in the Hopf-algebraic setting of comodule algebras, the invertibility of the canonical entwining $\psi$
is equivalent to the bijectivity of the antipode $S$. Then the left-coaction formula \eqref{leftco} reads
$_\mathcal{A}\delta=(S^{-1}\otimes\id)\circ\mathrm{flip}\circ\delta_\A$, and the above lemma is trivially true. 

The following theorem generalizes the fact that the pullback of the Ehresmann groupoid of a principal bundle 
is the Ehresmann 
groupoid of the pullback principal bundle.
\begin{theorem}\label{twofaith}
Let $\B\subseteq \mathcal{A}$ and $\B'\subseteq \mathcal{A}'$ be  $e$-coaugmented $\mathcal{C}$-Galois extensions  with 
bijective canonical 
entwinings. Assume that $\mathcal{A}'$ is faithfully flat as a left and  right $\B'$-module, and that the
coalgebra $\mathcal{C}$ is  cosemisimple. Then  every $\mathcal{C}$-equivariant algebra map 
$f:\mathcal{A}\rightarrow \mathcal{A}'$ restricts and 
corestricts to an algebra homomorphism $\B\rightarrow\B'$, and induces an isomorphism of  $\B'$-bimodules
$$
\B'{\otimes}_{\B}(\mathcal{A}\Box^{\mathcal{C}}\mathcal{A}){\otimes}_{\B}\B'\cong
 \mathcal{A'}\Box^{\mathcal{C}} \mathcal{A'}.
$$
\end{theorem}
\begin{proof}
Observe first that thanks to Lemma~\ref{cool}, Theorem \ref{faith} admits
the left $\mathcal{C}$-colinear right $\mathcal{B}$-linear (reversed) version. Now, 
using the associativity of the tensor product of $\B'$-bimodules, Theorem~\ref{faith} applied to $V=\mathcal{A}$, and the
reversed version of Theorem~\ref{faith} applied to $V=\mathcal{A}'$, we compute:
\begin{equation*}
\B'{\otimes}_{\B}(\mathcal{A}\Box^{\mathcal{C}}\mathcal{A}){\otimes}_{\B}\B'
            \cong(\B'{\otimes}_{\B}(\mathcal{A}\Box^{\mathcal{C}}\mathcal{A})){\otimes}_{\B}\B' 
            \cong (\mathcal{A'}\Box^{\mathcal{C}}\mathcal{A}){\otimes}_{\B}\B'
            \cong
 \mathcal{A'}\Box^{\mathcal{C}} \mathcal{A'}.\qedhere
\end{equation*}
\end{proof}

\subsection{Principal coactions}

\begin{definition}
\label{def.principal2}\cite{chg}
Let $\B\subseteq \mathcal{A}$ be an $e$-coaugmented $\mathcal{C}$-Galois extension. 
We call such an extension a
{\em principal $\mathcal{C}$-extension}
 iff
\begin{enumerate}
\vspace*{-.5mm}\item[(1)] $\psi:\mathcal{C}\otimes \mathcal{A}{\rightarrow} \mathcal{A}
\otimes \mathcal{C}$, $c\otimes a\mapsto
can(can^{-1}(1\otimes c)a)$ is bijective (invertibility of the \emph{canonical entwining}),
\item[(2)]
there exists a left $\B$-linear right $\C$-colinear splitting of the multiplication map  $\B\otimes \mathcal{A}\to \mathcal{A}$
(\emph{equivariant projectivity}). 
\end{enumerate}
\end{definition} 

Next, let us consider $\mathcal{A}\otimes\mathcal{A}$
as a $\mathcal{C}$-bicomodule via the right coaction
$
{\rm id}\otimes\delta_{\mathcal{A}}$ and the left coaction $
_{\mathcal{A}}\delta\otimes {\rm id}$, and $\mathcal{C}$ as a $\mathcal{C}$-bicomodule 
via its comultiplication.
\begin{definition}\label{str}
 A \emph{strong connection} is a $\mathcal{C}$-bicolinear map
$\ell:\mathcal{C}\rightarrow \mathcal{A}\otimes\mathcal{A}$ such that $\ell(e)=1\otimes 1$ and
$m\circ\ell=\varepsilon$, where $m$ and $\varepsilon$ stand for the multiplication in $\mathcal{A}$ 
and the counit of $\mathcal{C}$, respectively.
\end{definition}
\noindent
It is clear that the above definition of a strong connection coincides with its Hopf-algebraic 
counterpart by choosing $e=1$ (see the beginning of
Section~1).

\begin{theorem}\label{main}
Let $\B\subseteq \mathcal{A}$ and $\B'\subseteq \mathcal{A}'$ be principal $\mathcal{C}$-extensions, 
and $V$ a finite-dimensional
left $\C$-comodule.
 Then  every $\mathcal{C}$-equivariant  algebra map $f\colon\mathcal{A}\rightarrow \mathcal{A}'$ restricts 
and corestricts to an algebra
homomorphism $\B\to\B'$, and induces an
 isomorphism of  finitely generated projective left $\mathcal{\B}'$-modules 
$$
\B'{\otimes}_{\B}(\mathcal{A}\Box^{\mathcal{C}} V)\cong
 \mathcal{A'}\Box^{\mathcal{C}} V
$$
that is natural in~$V$. 
In particular, the induced map $f_*\colon K_0(\B)\to K_0(\B')$
satisfies 
$$
f_*([\A\Box^\C V])=[\A'\Box^\C V].
$$
\end{theorem}

\begin{proof}
Note first that combining \cite[Lemma~2.2]{chg} with \cite[Lemma~2.3]{chg}
 implies that a principal $\C$-extension always admits
a strong connection:
\begin{equation}\label{strongell}
\ell:
\mathcal{C}\longrightarrow \mathcal{A}\otimes \mathcal{A},\quad 
\sum_{\mu}a_{\mu}\otimes r_{\mu}(c):=\ell(c)=:
\ell\left( c\right)^{\langle  1\rangle}\otimes  \ell\left(c\right)^{\langle  2\rangle}
\end{equation}
(summation suppressed), where $\{a_\mu\}_\mu$ is a basis of~$\A$.
Given a unital linear functional $\varphi: \mathcal{A}\rightarrow \mathbb{C}$, one can construct \cite{piece} a left $\B$-linear 
map  $\sigma: \mathcal{A}\rightarrow \B$,
\begin{align}\label{sigma}
\sigma(a):= a_{(0)}\ell\left( a_{(1)} \right)^{\langle  1\rangle}\varphi\left( \ell \left(a_{(1)}\right)^{\langle  2\rangle}\right),
\end{align}
such that  $\sigma(b)=b$ for all $b\in \B$.
For a finite-dimensional left $\mathcal{C}$-comodule $V$ with a basis $\{v_i\}_i$, 
we define the coefficient matrix of the coaction 
$\varrho\colon V\rightarrow \mathcal{C}\otimes V$ with
respect to $\{v_i\}_i$ by
$
 \varrho(v_{i})=: \sum_{j}c_{ij}\otimes v_{j}
$.
By \cite[Theorem~3.1]{chg}, we can now apply 
\eqref{strongell} and \eqref{sigma} to the $c_{ij}$ to obtain a finite-size (say $N$) 
idempotent  matrix $e$ with entries
\begin{equation}
e_{(\mu, i) (\nu, j)}:=\sigma(r_{\mu}(c_{ij})a_{\nu})\in \B
\end{equation}
such that $\mathcal{A}\Box^{\mathcal{C}}V\cong\B^Ne$ as left $\B$-modules.
Consequently, $\mathcal{A}\Box^{\mathcal{C}}V$ is  finitely generated projective, and its class in $K_{0}(\B)$ can be 
represented by~$e$. 

Since $f: \mathcal{A}\rightarrow \mathcal{A}'$ satisfies the assumptions of Lemma~\ref{cool},
\begin{equation}\label{elp}
\ell':=(f\otimes f)\circ\ell : \mathcal{C}\longrightarrow\mathcal{A}'\otimes \mathcal{A}'
\end{equation}
is a strong connection on $\mathcal{A}'$.
Next,
we choose
bases $\{ a_{\mu}\mid \mu\in J\}$ and $\{ a'_{\mu}\mid \mu\in J'\}$ of $\mathcal{A}$ and $\mathcal{A}'$ 
respectively  in such a way that 
\begin{equation*}
\{ a'_{\mu}=f(a_{\mu})\, \mid\,\mu\in I \} \text{ is a basis of }f(\mathcal{A}) \quad \text{and}\quad 
\{a_{\mu}\,\mid\, \mu\not\in I\}\text{ is a basis of }\ker f.
\end{equation*}
Under the above choices, using \eqref{strongell} and \eqref{elp}, we compute
\begin{equation}
\sum_{\mu}a'_{\mu}\otimes r'_{\mu}(c):=\ell'(c)
=\sum_{\mu}f(a_{\mu})\otimes f(r_{\mu}(c))=\sum_{\mu}a'_{\mu}\otimes f(r_{\mu}(c)).
\end{equation}
Thus we obtain 
$
r'_{\mu}(h)= f(r_{\mu}(h)) \text{ for all }\mu\in I
$.

Now we choose a unital functional $\varphi'$ on $\mathcal{A}'$, and take  $\varphi:=\varphi'\circ f$. 
For $\sigma'$ produced 
from $\varphi'$ and 
$\ell'$ as in (\ref{sigma}), we check that the diagram
\begin{equation}
\xymatrix @C=1.5pc @R=0.5pc{                    & \mathbb{C} &           \\ & & \\
\mathcal{A}\ar[ruu]^-{\varphi}\ar[rr]^-{f}\ar@/_2em/[dd]_-{\sigma}&  
&\mathcal{A}'\ar[luu]_-{\varphi'}\ar@/^2em/[dd]^-{\sigma'}\\
\cup\hspace{-0.31em}\shortmid & & \cup\hspace{-0.31em}\shortmid \\
B\ar[rr]^-{f\mid_{B}}&&B'}
\end{equation}
commutes by the following calculation. First we compute
\begin{align}
\sigma'(a') &= a'_{(0)}\ell'( a'_{(1)})^{\langle  1\rangle}\varphi'\big( \ell'(a'_{(1)})^{\langle  2\rangle}\big)\nonumber\\
                 &= a'_{(0)}f\big(\ell ( a'_{(1)})^{\langle  1\rangle}\big)\varphi'\big( f(\ell (a'_{(1)})^{\langle  2\rangle})\big)
\nonumber\\
                 &= a'_{(0)}f\big(\ell( a'_{(1)})^{\langle  1\rangle}\big)\varphi\big(\ell (a'_{(1)})^{\langle  2\rangle}\big).
\end{align}
Next we plug in $a'=f(a)$ to get
\begin{align}
\sigma'(f(a)) &= f(a)_{(0)} f\left(\ell\left(f( a)_{(1)}\right)^{\langle  1\rangle}\right)
\varphi\left(\ell (f(a)_{(1)})^{\langle  2\rangle}\right)\nonumber\\
                    &= f\left(a_{(0)}\right) f\left(\ell\left( a_{(1)}\right)^{\langle  1\rangle}\right)
\varphi\left(\ell (a_{(1)})^{\langle  2\rangle}\right)\nonumber\\
                    &= f\left(a_{(0)}\ell\left( a_{(1)}\right)^{\langle  1\rangle}
\varphi\left(\ell (a_{(1)})^{\langle  2\rangle}\right)\right)\nonumber\\
                    &=f(\sigma(a)).
                 \end{align}
Hence
\begin{align}
f\left(e_{(\mu, i) (\nu, j)}\right)&=f\left(\sigma \left(r_{\mu}(c_{ij})a_{\nu}\right)\right)
=\sigma'\left(f\left(r_{\mu}(c_{ij})a_{\nu}\right)\right)\nonumber\\
&=\sigma'\left(f\left(r_{\mu}(c_{ij})\right)f\left(a_{\nu}\right)\right)=\sigma'\left(r'_{\mu}(c_{ij})a'_{\nu}\right).
\end{align}
Note that $f\left(e_{(\mu, i) (\nu, j)}\right)$ is zero for $\nu\not\in I$ because then $f(a_\nu)=0$. 

Furthermore,
applying \cite[Theorem~3.1]{chg} to the strong connection $\ell'$, the basis $\{a'_\mu\}_\mu$, 
and the matrix coefficients~$c_{ij}$,
for all $\mu,\nu\in I$, $i,j\in\{1,\ldots,\dim V\}$, we obtain
\begin{equation}
\sigma'\left(r'_{\mu}(c_{ij})a'_{\nu}\right)=:e'_{(\mu, i) (\nu, j)}\,,
\end{equation}
where the $e'_{(\mu, i) (\nu, j)}$ are the entries of an idempotent matrix $e'$ such that
 $\B'^{N'}e'\cong \A'\Box^\C V$ as left $\B'$-modules. Thus, in the block matrix notation, 
we arrive at the following crucial equality
\begin{equation}
f(e)=\left(\begin{array}{cc}
                   e' & 0\\
                  d & 0
        \end{array}\right),
\end{equation}
where $d$ is unspecified.
Now, taking into account that $f(e)$ is an  idempotent matrix, we derive the equality
$de'=d$, which
allows us to verify that
\begin{equation}
\left(\begin{array}{cc}
                   e' & 0\\
                  d & 0
        \end{array}\right)=\left(\begin{array}{cc}
                   1 & 0\\
                  d & 1
        \end{array}\right)\left(\begin{array}{cc}
                   e' & 0\\
                  0 & 0
        \end{array}\right)\left(\begin{array}{cc}
                   1 & 0\\
                  d & 1
        \end{array}\right)^{-1}.
\end{equation}
Hence the corresponding finitely generated projective left $\B'$-modules are isomorphic: 
\begin{equation}
\B'\otimes_\B(A\Box^\C V)\cong\B'\otimes_\B(B^Ne)\cong\B'^Nf(e)\cong \B'^{N'}e'\cong \A'\Box^\C V.
\end{equation}
In particular, $\left(f\hspace{-0.25em}\mid_{\B}\right)_{*}[e]:=[f(e)]=[e']\in K_{0}(\B')$.
\end{proof}

\subsection{The Hopf-algebraic case revisited}

We end this section by arguing that Theorem~\ref{main} specializes to 
Theorem~\ref{hopf} in the Hopf-algebraic setting.
First, observe that the lacking assumption of the principality of~$\A'$ in Theorem~\ref{hopf} is redundant.  
Indeed, if $\ell$ is a strong connection
on $\A$ and $f\colon\A\to\A'$ is an $\cH$-equivariant algebra homomorphism, 
then $(f\otimes f)\circ\ell$ is immediately a strong connection on~$\A'$.
Furthermore, the bijectivity of 
the antipode $S$ is equivalent
to the invertibility of the canonical entwining, and the coaugmentation is readily provided by $1\in\cH$.
Finally, as is explained at the beginning of Section~1, the existence of a strong connection implies both the bijectivity of the canonical map
and equivariant projectivity.

\section{Pulling back noncommutative vector bundles associated with free actions of compact quantum groups}

Let $(H,\Delta)$ be a compact quantum group~\cite{wor}. Let $A$ be a unital C*-algebra and
\mbox{$\delta_{A}:A\rightarrow A\otimes_{\mathrm{min}}H$} an
injective  unital $*$-homomorphism, where $\otimes_{\mathrm{min}}$ denotes the minimal tensor product of C*-algebras.
 We call $\delta_{A}$ a \emph{right coaction}
of $H$ on $A$ (or a \emph{right action of the compact quantum group on a compact quantum space}) \cite{p-p95} iff
\begin{enumerate}
\item[(1)]
$(\delta_{A}\otimes\mathrm{id}_{H})\circ\delta_{A}=
(\mathrm{id}_{H}\otimes\Delta)\circ\delta_{A}$
(coassociativity),
\item[(2)]
$\{\delta_{A}(a)(1\otimes h)\;|\;a\in A,\,h\in H\}^{\mathrm{cls}}=
A\underset{\mathrm{min}}{\otimes}H$ (counitality).
\end{enumerate}
Here ``cls'' stands for ``closed linear span''. Furthermore,
a coaction $\delta_{A}$ is called \emph{free} \cite{ell}  iff
\begin{equation}
\{(a\otimes 1)\delta_{A}(\tilde{a})\;|\;a, \tilde{a}\in A\}^{\mathrm{cls}}=
A\underset{\mathrm{min}}{\otimes}H.
\end{equation}

Given a compact quantum group $(H,\Delta)$, we denote by
 $\cO (H)$ its dense cosemisimple
Hopf $*$-subalgebra spanned by the matrix coefficients of
irreducible (or finite-dimensional) unitary corepresentations~\cite{wor}.
In the same spirit, we define the
\emph{Peter-Weyl subalgebra}  of~$A$ 
(see~\cite{free}, cf.~\cite{p-p95,s-pm11}) as
\begin{equation}
\cP_H(A):=\left\{\,a\in A\,| \,\delta_{A}(a)\in A\otimes\cO (H)\,\right\}.
\end{equation}
It follows from Woronowicz's definition of a compact quantum group that
the left and right coactions of $H$ on itself by the  comultiplication are free.
Also, it is easy to check that
 \mbox{$\cP_H(H)=\cO (H)$}, and
 $\cP_H(H)$ is a right  $\cO (H)$-comodule  via the restriction-corestriction of~$\delta_{A}$~\cite{free}.
Its coaction-invariant subalgebra coincides with the fixed-point subalgebra
\begin{equation}
B:=\{b\in A\;|\;\delta_{A}(b)=b\otimes 1\}.
\end{equation}
A fundamental result concerning Peter-Weyl comodule algebras  is
that the freeness of an action of a compact quantum group $(H,\Delta)$ on a unital C*-algebra $A$ is
\emph{equivalent} to the principality of the Peter-Weyl $\cO (H)$-comodule 
algebra~$\mathcal{P}_H(A)$~\cite{free}.
The result bridges algebra and analysis allowing us to conclude from Theorem~\ref{hopf} the following
crucial claim.

\begin{theorem}\label{cor}
Let $(H,\Delta)$ be a compact quantum group, let $A$ and $A'$ be  $(H,\Delta)$-C*-algebras,
$B$ and $B'$ the corresponding fixed-point subalgebras,
and
$f:A\to A'$  an equivariant *-homomorphism. Then, if the coaction of $(H,\Delta)$ on $A$ is free and $V$
is a representation  of $(H,\Delta)$, the following left $B'$-modules are isomorphic
$$
B'\otimes_{B}\big(\mathcal{P}_H(A)\Box^{\mathcal{O}(H)} V\big)\cong 
\mathcal{P}_H(A')\Box^{\mathcal{O}(H)} V.
$$
In particular, if $V$ is finite dimensional, then the induced map 
\mbox{$f_*\colon K_0(B)\to K_0(B')$}
satisfies 
$$
f_*([\mathcal{P}_H(A)\Box^{\mathcal{O}(H)} V])
=[\mathcal{P}_H(A')\Box^{\mathcal{O}(H)} V].
$$
\end{theorem}

\begin{proof}
 Note that, since $\mathcal{O}(H)$ is cosemisimple, any comodule is a direct sum of finite-dimensional
comodules, so that it suffices to prove the theorem for finite-dimesional representations of $(H,\Delta)$.
By \cite{free}, the freeness  of the $(H,\Delta)$-action is equivalent to principality of the Peter-Weyl
comodule algebra~$\mathcal{P}_H(A)$, which (as explained at the beginning of Section~1) is
tantamount to the existence of a strong connection:
$\ell: \mathcal{O}(H)\rightarrow \mathcal{P}_H(A)\otimes \mathcal{P}_H(A)$. 
Now the claim follows from Theorem~\ref{hopf} applied to the case  
$\mathcal{A}=\mathcal{P}_H(A)$, $\mathcal{H}=\mathcal{O}(H)$, and 
$\B=B$.
\end{proof}

Recall that the existence of a strong connection implies equivariant projectivity, which (by \cite{sch-sch}) is 
equivalent to faithful flatness. Combining this with the cosemisimplicity of~$\mathcal{O}(H)$, we can view
the above Theorem~\ref{cor} as a specialization of Theorem~\ref{faith}.

\subsection{Iterated equivariant noncommutative join construction}

Let $G$ be a topological group. Recall that the join of two $G$-spaces is again a $G$-space for the diagonal action of~$G$.
It is this action that is natural for topological constructions. A straightforward generalization of the diagonal action to 
the realm of compact quantum groups $(H,\Delta)$ acting on C*-algebras would require
that there exists a *-homomorphism $H\otimes_{\rm min}H\to H$ extending the algebraic multiplication map, which is typically not the case.
However, when taking the join $X*G$ of a $G$-space $X$ with $G$, the diagonal action of $G$ on $X*G$ can be gauged to the action on 
the $G$-component
alone.  Thus we obtain an equivalent classical construction  that is amenable to noncommutative deformations. (See \cite{dhh15} for details,
cf.~\cite{nv10} for an alternative approach.) 

\begin{definition}\cite{dhh15,free}\label{joindef} 
For any compact quantum group $(H,\Delta)$ acting  on a unital C*-algebra~$A$  via
$\delta_A\colon A\to A\otimes_{\mathrm{min}}H$, we define
its \emph{equivariant join with $H$} to be the unital C*-algebra
\vspace*{-2mm}
$$
A\overset{\delta_A}{\circledast} H:=\left\{\!f \in C([0,1],A)\!\!\underset{\mathrm{min}}{\otimes} 
\!\!H=C([0,1],A\!\!\underset{\mathrm{min}}{\otimes}\! \!H)\;\Big{|}\;
	f(0) \in \mathbb{C}\! \otimes\! H,\;f(1) \in \delta_A(A)\!\right\}.
$$
\end{definition}
\begin{theorem}\label{2}\cite{free}
Let  $(H,\Delta)$ be a compact quantum group acting  on a unital C*-algebra~$A$.
Then the *-homomorphism
$$
\mathrm{id}\!\otimes\!\Delta\colon\; C([0,1],A)\! \underset{\mathrm{min}}{\otimes} \!H\;\longrightarrow\;
C([0,1],A)\! \underset{\mathrm{min}}{\otimes} \!H\! \underset{\mathrm{min}}{\otimes} \!H
$$
restricts and corestricts to a *-homomorphism
$$
\delta_\Delta\colon A\circledast^{\delta_A} H\longrightarrow (A\circledast^{\delta_A} H)
\! \underset{\mathrm{min}}{\otimes} \!H
$$
defining an action of $(H,\Delta)$  on  \mbox{$A\!\circledast^{\delta_{A}}\!H$}. If the action of
$(H,\Delta)$  on  $A$ is free, then so is the above action of $(H,\Delta)$ on $A\!\circledast^{\delta_{A}}\!H$.
\end{theorem}

Note that
any equivariant *-homomorphism $F:A\to A'$ of $(H,\Delta)$-C*-algebras
induces an $(H,\Delta)$-equivariant *-homomorphism 
$A\!\circledast^{\delta_{A}}\! H\rightarrow A'\!\circledast^{\delta_{A'}}\! H$.
Indeed, since the *-homomorphism $F$ is equivariant, the *-homomorphism
\[
\id\otimes F\otimes\id:C([0,1])\underset{\mathrm{min}}{\otimes}A\underset{\mathrm{min}}{\otimes} H
\longrightarrow C([0,1])\underset{\mathrm{min}}{\otimes}A'\underset{\mathrm{min}}{\otimes}H
\]
restricts and corestricts to an equivariant
*-homomorphism \mbox{$f\!:\!A\!\circledast^{\delta_{A}}\! H\!\to\! A'\!\circledast^{\delta_{A'}}\! H$}. 
It is clear  that the composition of equivariant maps induces the composition of the induced equivariant maps. 
We refer to this fact as the \emph{naturality} of the noncommutative equivariant join with~$H$.   

Starting from $A=H$, we can iterate Definition~\ref{joindef} finitely many times. According to 
Theorem~\ref{2}, the thus $n$-times iterated equivariant join $A_n$ comes equipped
 with a free $(H,\Delta)$-action.
The following lemma gives a construction of an equivariant map $A_n\to A_1$ allowing us later 
on to apply Theorem~\ref{cor}. 
\begin{lemma}\label{char}
Let $(H,\Delta)$ be a compact quantum group such that the C*-algebra $H$ admits a character. 
Then,  for any $n\in\mathbb{N}\setminus\{0\}$, there exists an 
equivariant *-homomorphism
from the $n$-iterated equivariant join $A_n$ to $A_1:=H\!\circledast^{\Delta}\!H$.
\end{lemma}
\begin{proof}
If $\chi:H\to\mathbb{C}$ is a character,
then
\[
f_\chi:=\mathrm{ev}_{\frac{1}{2}}\otimes\chi\otimes \id:H\!\circledast^{\Delta}\! H\longrightarrow H
\]
is an equivariant
*-homomorphism. 
More generally, applying  $f_\chi$ to the leftmost  factor in the $n$-times iterated equivariant join $A_n$,
 we obtain an equivariant map to the $(n-1)$-times iterated equivariant join $A_{n-1}$:
\[
\id_{C([0,1])}^{\otimes\, (n-1)}\otimes f_\chi\otimes\id_H^{\otimes\,(n-1)}\colon A_n\longrightarrow A_{n-1}.
\]
Composing all these maps, we obtain an equivariant map $A_n\to A_1$, as desired.
\end{proof}

\subsection{Noncommutative tautological quaternionic line bundles and their duals}

To fix notation, let us begin by recalling the definition of $SU_{q}(2)$~\cite{w-sl87}.
We take $H=C(SU_{q}(2))$ to be
the universal unital C*-algebra 
generated by $\alpha$ and $\gamma$ subject to the relations
\begin{equation}
\label{suq2_relations}
\alpha \gamma = q \gamma \alpha ,\,  \;\ \alpha \gamma^* = q \gamma^* \alpha ,\,  \;\ \gamma \gamma^* 
= \gamma^* \gamma,\, \;\
\alpha^* \alpha + \gamma^* \gamma = 1, \,  \;\ \alpha \alpha^* + q^2 \gamma \gamma^* = 1, \, 
\end{equation}
where $0<q\leq 1$. The coproduct $\Delta: C(SU_{q}(2))\to C(SU_{q}(2))\otimes_{\rm min}C(SU_{q}(2))$ is given by the formula
\[
\Delta
\begin{pmatrix}
\alpha& -q\gamma^*\\
\gamma & \alpha^*
\end{pmatrix}
=
\begin{pmatrix}
\alpha\otimes 1& -q\gamma^*\otimes 1\\
\gamma\otimes 1 & \alpha^*\otimes 1
\end{pmatrix}
\begin{pmatrix}
1\otimes\alpha& 1\otimes -q\gamma^*\\
1\otimes\gamma & 1\otimes\alpha^*
\end{pmatrix}.
\]
We call the matrix 
\[
u:=\left( \begin{matrix} \alpha& -q\gamma^*\\
\gamma & \alpha^* \end{matrix} \right)
\]
 the \emph{fundamental representation} matrix of~$SU_q(2)$. 
The counit $\varepsilon$ and the antipode $S$ of the Hopf algebra $\mathcal{O}(SU_{q}(2))$ are respectively given by the formulas
\[
\varepsilon
\begin{pmatrix}
\alpha& -q\gamma^*\\
\gamma & \alpha^*
\end{pmatrix}
=
\begin{pmatrix}
1& 0\\
0& 1
\end{pmatrix},
\qquad
S
\begin{pmatrix}
\alpha& -q\gamma^*\\
\gamma & \alpha^*
\end{pmatrix}
=
\begin{pmatrix}
\alpha^*& \gamma^*\\
-q\gamma & \alpha
\end{pmatrix}.
\]
Combining the antipode $S$
with the matrix transposition, we obtain the matrix 
\[
u^{\vee}:=S(u^T)=\left( \begin{matrix} \alpha^*& -q\gamma\\
\gamma^* & \alpha \end{matrix} \right)
\]
of the representation \emph{contragredient} to the fundamental representation. Note, however, that
\[\label{equivalence}
\left( \begin{matrix} \alpha^*& -q\gamma\\
\gamma^* & \alpha \end{matrix} \right)=\left( \begin{matrix} 0& -q\\
1 & 0 \end{matrix} \right)\left( \begin{matrix} \alpha& -q\gamma^*\\
\gamma & \alpha^* \end{matrix} \right)\left( \begin{matrix} 0& -q\\
1 & 0 \end{matrix} \right)^{-1},
\]
which means that $u^{\vee}$ and $u$ are equivalent as complex  representations  for arbitrary~$q$. 

Motivated by the classical situation, we denote the $n$-iterated equivariant join of $C(SU_q(2))$ 
by $C(S^{4n+3}_q)$. Note that, except for $n=0$, our quantum spheres are different from 
Vaksman-Soibel'man
quantum spheres~\cite{vs91}. Indeed, it follows immediately from the defining relations 
 of the C*-algebra of any Vaksman-Soibel'man
quantum sphere that its space of characters is always a circle.
In contrast, the space of characters of  $C(S^{4n+3}_q)$ always contains
the torus $\mathbb{T}^{n+1}$ appearing as the Cartesian product of the circles of characters of
 copies of~$C(SU_q(2))$
in the iterated equivariant join. 

The quaternionic projective space $\mathbb{HP}^n$ is defined as the quotient space\linebreak 
\mbox{$(\mathbb{H}^{n+1}\setminus \{0\})/\mathbb{H}^\times$} with respect to 
the right multiplication by non-zero quaternions  of non-zero quaternionic column vectors in 
$\mathbb{H}^{n+1}$.
Observe that the above quotienting
restricts to quotienting the space of unit vectors $S^{4n+3}\subset \mathbb{H}^{n+1}$ by 
the right multiplication by unit quaternions. The latter can be identified with the group $SU(2)$.
Since, for $q=1$, the action of $(C(SU_q(2)),\Delta)$ on $C(S^{4n+3}_q)$
is equivalent to the above standard right $SU(2)$-action  on $S^{4n+3}$ with $\mathbb{HP}^n$
 as the space of orbits,
we propose the following definition.
\begin{definition}
We define the C*-algebra of the $n$-th \emph{quantum quaternionic projective space} 
as the fixed-point subalgebra
\begin{align*}
C(\mathbb{HP}^n_q):=C(S^{4n+3}_q/SU_q(2)):=&
C(S^{4n+3}_q)^{SU_q(2)}\\
:=& \big\{b\in C(S^{4n+3}_q)\;|\; \delta_{C(S^{4n+3}_q)}(b)
= b\otimes 1\big\}.
\end{align*}
\end{definition}

For the sake of brevity, let us denote the Hopf algebra $\mathcal{O}(C(SU_q(2)))$ by\linebreak
 $\mathcal{O}(SU_q(2))$,
and the $\mathcal{O}(SU_q(2))$-comodule algebra $\mathcal{P}_{C(SU_q(2))}(C(S^{4n+3}_q))$ 
by\linebreak
 $\mathcal{P}_{SU_q(2)}(S^{4n+3}_q)$.
Since the $(C(SU_q(2)),\Delta)$-action on $C(S^{4n+3}_q)$ is free, for any 
finite-dimensional corepresentation $V$ of $\mathcal{O}(SU_q(2))$,
the associated left 
$C(\mathbb{HP}^n_q)$-module $\mathcal{P}_{SU_q(2)}(S^{4n+3}_q)\Box^{\mathcal{O}(SU_q(2))} V$ 
is finitely generated projective by~\cite[Theorem~1.2]{dy13}.

Next, let $V$ and $V^{\vee}$ be   left corepresentations of $\mathcal{O}(SU_q(2))$ on $\mathbb{C}^2$ given
respectively by the fundamental and  contragredient representation matrices $u$ and $u^{\vee}$.
Now consider the corresponding associated $C(\mathbb{HP}^n_q)$-modules
\begin{align}
&\tau_{{}_{\mathbb{HP}^n_q}}:=\mathcal{P}_{SU_q(2)}(S^{4n+3}_q)\Box^{\mathcal{O}(SU_q(2))} V^{\vee},\\
&\tau^*_{{}_{\mathbb{HP}^n_q}}:=\mathcal{P}_{SU_q(2)}(S^{4n+3}_q)\Box^{\mathcal{O}(SU_q(2))} V.
\end{align}
For $q=1$, under the standard embedding $\mathbb{C}\subset \mathbb{H}$, 
the first module is the section module of the tautological quaternionic line bundle, 
and the second module is the section module
of its complex dual. 
Hence we refer to $\tau_{{}_{\mathbb{HP}^n_q}}$ as
the section module of the \emph{noncommutative tautological quaternionic line bundle}, and   to
 $\tau^*_{{}_{\mathbb{HP}^n_q}}$ as the section module  of its complex  dual.
It follows from  the functoriality of the cotensor product and the equivalence \eqref{equivalence} that the associated modules 
$\tau_{{}_{\mathbb{HP}^n_q}}$ and $\tau^*_{{}_{\mathbb{HP}^n_q}}$ are isomorphic.

Note that restricting coefficients from $\mathbb{H}$ to $\mathbb{C}$ is compatible with dualization. 
This means that, for every right $\mathbb{H}$-module (vector space) $W$, the  map from the complex vector space 
${\rm Hom}_{\mathbb{C}}(W, \mathbb{C})$ to the left $\mathbb{H}$-module (vector space)
${\rm Hom}_{\mathbb{H}}(W, \mathbb{H})$
defined by 
\[
\alpha\longmapsto (w\mapsto \alpha(w) - \alpha(wj)j)
\]
is an isomorphism of complex vector spaces. Here $j$ is the quaternionic imaginary 
unit anticommuting with the complex imaginary unit, and the vector-space structures on Hom-spaces
are given by the left multiplication on  values.
In this way, we identify (as complex vector bundles) the complex dual of a right quaternionic vector bundle with its quaternionic dual. 

The following result is our main application of Theorem~\ref{cor}.
\begin{theorem}\label{2.5}
For any $n\in\mathbb{N}\setminus\{0\}$ and $0<q\leq 1$, the finitely generated projective left
$C(\mathbb{HP}^n_q)$-modules $\tau_{{}_{\mathbb{HP}^n_q}}$ and $\tau^*_{{}_{\mathbb{HP}^n_q}}$
are \emph{not} stably free, i.e.,  the noncommutative tautological quaternionic line bundle and its dual
 are \emph{not} stably trivial as noncommutative complex vector bundles.
\end{theorem}
\begin{proof} 
It is explained in \cite[Section~3.2]{bu} how the non-vanishing of an index pairing computed in \cite{ind} for the fundamental
representation matrix $u$ implies that $\tau^*_{{}_{\mathbb{HP}^1_q}}$ is not stably free.

Furthermore, for $H= C(SU_q(2))$  we have a circle of characters, so that Lemma~\ref{char} yields an equivariant *-homomorphism
$C(S^{4n+3}_q)\to C(S^7_q)$. This allows us to apply Theorem~\ref{cor} to conclude that 
\begin{align}
&[\tau^*_{{}_{\mathbb{HP}^n_q}}]=2[1]\in K_0(C(\mathbb{HP}^n_q))\;\Rightarrow\;
[\tau^*_{{}_{\mathbb{HP}^1_q}}]=2[1]\in K_0(C(\mathbb{HP}^1_q)).
\end{align}

As this contradicts the stable non-triviality of  $\tau^*_{{}_{\mathbb{HP}^1_q}}$, we infer that the  left
$C(\mathbb{HP}^n_q)$-module $\tau^*_{{}_{\mathbb{HP}^n_q}}$ 
is \emph{not} stably free. Finally, the fact that \eqref{equivalence} induces an isomorphism between $\tau^*_{{}_{\mathbb{HP}^n_q}}$ and  
$\tau_{{}_{\mathbb{HP}^n_q}}$
completes the proof.
\end{proof}

Now it follows from \cite[Proposition~3.2]{bu} that the noncommutative Borsuk-Ulam type~2 conjecture \cite{bu} holds for $C(S^{4n+3}_q)$:
\begin{corollary}
For any $n\in\mathbb{N}\setminus\{0\}$, there does \emph{not} exist
	a $(C(SU_q(2)),\Delta)$-equivariant *-homomorphism $C(SU_q(2))\to C(S^{4n+3}_q)\!\circledast^{\delta_{C(S^{4n+3}_q)}}\! C(SU_q(2))$.
\end{corollary}
\noindent
In other words, there does \emph{not} exist
	an $SU_q(2)$-equivariant continuous map from $S^{4n+7}_q$ to~$S^{3}_q$.

\section*{Acknowledgments} 
This work was partially supported by  NCN grant 
2012/06/M/ST1/00169. It is a pleasure to thank Kenny De Commer for pointing to us~\eqref{equivalence}.


\end{document}